\documentclass[reqno,makeidx,a4paper]{amsart}  
\usepackage{enumerate,amsmath,amssymb,amsthm}     
\usepackage[linkbordercolor={1 1 1}]{hyperref}
\usepackage{lscape}
\usepackage{enumitem}
\usepackage[pdftex]{graphicx}
\usepackage{array}
\usepackage{color}

\theoremstyle{plain}      
 
\newtheorem{thm}{Theorem}
\newtheorem{theorem}[thm]{Theorem}     
     
\newtheorem{corollary}[thm]{Corollary}     
     
\newtheorem{lemma}[thm]{Lemma}     
     
\newtheorem{proposition}[thm]{Proposition}     
    
\newtheorem{claim}[thm]{Claim}   
\theoremstyle{remark}      
  
\newtheorem{remark}[thm]{Remark}
     
\theoremstyle{definition}

\makeindex 

\DeclareMathAlphabet{\doba}{U}{msb}{m}{n}

\def\cA{\mathcal{A}}
\def\cB{\mathcal{B}}

\def\cI{\mathcal{I}}
\def\cJ{\mathcal{J}}

\def\cM{\mathcal{M}}
\def\cN{\mathcal{N}}

\def\cQ{\mathcal{Q}}
\def\cR{\mathcal{R}}
\def\cS{\mathcal{S}}
\def\cT{\mathcal{T}}
\def\cU{\mathcal{U}}

\def\Sym{\mathrm{Sym}}

\def\no{{|}}
\def\fac{{!}}

\def\tr{{\mathrm{Tr\,}}}
\def\Ric{{\mathrm{Ric}}}
\def\Riem{{\mathrm{Riem}}}
\def\scal{\mathrm{scal\,}}
\def\di{{\rm d}}

\let\<\langle 
\let\>\rangle

\newcommand{\definedas}{\mathrel{\raise.095ex\hbox{\rm :}\mkern-5.2mu=}}

\begin{document}     


\title
{A detailed proof of a theorem of  Aubin}
 
\author{Farid Madani} 
\address{Fakult\"at f\"ur  Mathematik \\ 
Universit\"at Regensburg \\
93053 Regensburg \\  
Germany}
\email{Farid.Madani@mathematik.uni-regensburg.de}

\begin{abstract}
In this note we give a detailed proof of a theorem of 
Aubin, namely \cite[Th\'eor\`eme 5]{Aub3}.
\end{abstract}


\maketitle


\section{Introduction}

The goal of this note is to give a detailed proof of the following theorem of Aubin.
\begin{theorem}[Aubin \cite{Aub3}]\label{main thm}
Let $(M,g)$ be a compact Riemannian manifold of dimension $n$. 
Assume that, in a geodesic normal coordinate chart around a point 
$x\in M$, $\det(g)=1+O(r^N)$, with $N$  sufficiently large, where 
$r=d(x,\cdot)$. Let  $\omega$ denote 
$\omega:=\inf \{k\in \mathbb N : |\nabla^k Weyl_g(x)|\neq0\}$. 
If  $|\nabla^{\omega}\scal(x)|=0$, then $\Delta^{\omega+1}\scal(x)<0$. 
In particular,  $\int_{\partial B_r(x)}\scal \di \sigma<0$, for $r$ sufficiently small.
\end{theorem} 
For the particular values of $\omega\in\{1,2\}$, Theorem \ref{main thm} is proven by 
Hebey and Vaugon \cite{HV}. The case $\omega=3$ was studied by L. Zhang, as mentioned 
in \cite{Aub3}. 

Theorem \ref{main thm} is fundamental in questions related to the compactness of 
the Yamabe equation solutions and the equivariant Yamabe problem. 
For example, in \cite{Mad2, Mad4}, the author used Theorem \ref{main thm} to prove 
the validity of the Hebey--Vaugon conjecture, when $\omega\leq \frac{n-6}{2}$ (more details 
are given in Section \ref{EYP}).

Aubin used Theorem \ref{main thm} in \cite{Aub3, Aub5, Aub4} 
in his study of the compactness of the set of solutions to the Yamabe 
equation.  He claimed the compactness of this set except for the round
 sphere. This claim is in a contradiction with the counter examples constructed 
by Brendle \cite{Bre} and by Brendle and Marques \cite{BMar}. Many mathematicians 
had serious concerns about the correctness of \cite{Aub3}, although it remained 
unclear which parts of \cite{Aub3} are correct. For example, the proof of 
\cite[Th\'eor\`eme 5]{Aub3} is only sketched and not given in all details, and thus 
it is difficult for a reader to check whether the result of this theorem holds or not. 
The aim of this note is to provide a rigorous
proof of \cite[Th\'eor\`eme 5]{Aub3}.

In our proof of Theorem \ref{main thm}, we did not follow exactly the strategy proposed 
by Aubin in \cite{Aub3}. However, there are certain common points. 
For example, the formula in Corollary \ref{equa cor} coincides with the last 
formula on page 279 in \cite{Aub3}. 

In order to facilitate the reading, we gathered all notation in 
Notation Index~\ref{notation}.

\section{The equivariant Yamabe problem}\label{EYP}
In this section, we present an application of Theorem \ref{main thm}. 
Let $(M,g)$ be a compact Riemannian manifold of dimension $n\geq 3$ and $G$ be a subgroup of the isometry 
group $\textrm{Isom}(M,g)$. The equivariant Yamabe problem consists in finding a 
$G$-invariant metric in the conformal class of $g$, with constant scalar curvature.  
Hebey and Vaugon proved that a sufficient condition to solve this problem is to prove that
the following conjecture holds. 
\paragraph{\bf Hebey--Vaugon conjecture}\emph{If $(M,g)$ is not 
conformal to the round sphere $S^n$ or if the action of $G$ has no fixed point, then the 
following strict inequality holds }
\begin{equation*}
\inf_{g'\in [g]^G} \frac{\int_M\scal_{g'}\di v_{g'}}{(\int_M\di v_{g'})^{\frac{n-2}{n}}}<
n(n-1)\omega_n^{2/n}(\inf_{x\in M}\textrm{card\,} (G\cdot x))^{2/n},
\end{equation*}
\emph{where $\omega_n$ is the volume of round sphere $S^n$}.

For $\omega\leq 2$ Hebey and Vaugon \cite{HV} proved Theorem \ref{main thm} and they 
used this result to prove the above conjecture. They also proved it 
for $\omega> \frac{n-6}{2}$, assuming the positive mass theorem.  

In \cite{Mad2, Mad4}, the author used Theorem \ref{main thm} to prove 
the validity of the Hebey--Vaugon conjecture for $\omega\leq \frac{n-6}{2}$, 
by constructing a local $G$-invariant positive test function,  
which involves the scalar curvature, where the negative sign of 
$\int_{\partial B_r(x)}\scal \di \sigma$ plays an important role. 
For more details, see \cite{HV}, \cite{Mad2, Mad4} 
and the references therein.

\section{The Hebey-Vaugon formulas}

Let us first mention some notation and convention used throughout this note. 
Our convention for the components of the Riemann and Ricci curvature tensors are: 
in a local normal coordinate chart, \index{Rie@$R^q_{jk\ell}X^j=\nabla_k\nabla_\ell X^q-\nabla_\ell\nabla_k X^q$|textbf} 
\begin{equation*}
R^q_{jk\ell}X^j=(\nabla_k\nabla_\ell X)^q-(\nabla_\ell\nabla_k X)^q, 
\quad (\nabla_{ij}\alpha)_k=(\nabla_{ji}\alpha)_k-R^\ell_{k ij}\alpha_\ell,
\end{equation*} 
for any vector field $X$ and  $1$-form $\alpha$. Moreover, we set
\index{Riem@$R_{ijk\ell}=g_{iq}R^q_{jk\ell}$|textbf} $R_{ijk\ell}=g_{iq}R^q_{jk\ell}$, 
\index{Ric@$\Ric_{ij}=R^\ell_{ i \ell j}$|textbf} $\Ric_{ij}=R^\ell_{ i \ell j}$.
For any $k$-covariant tensor $T$, we have:
\[\nabla_{q}T_{p_1\ldots p_{k}}:=(\nabla_{q}T)_{p_1\ldots p_{k}},\quad \Delta T_{p_1\ldots p_{k}}:=(g^{pq}\nabla_{pq}T)_{p_1\ldots p_{k}}\]
We define the symmetrization of the tensor $T$ by
\index{Sym@$\Sym T_{p_1\ldots p_k}= \underset{\sigma\in\mathfrak{S}(k)}{\sum} T_{p_{\sigma(1)}\ldots p_{\sigma(k)} }$|textbf}
\[\Sym T_I:= \underset{p_1,\ldots, p_k}\Sym T_{p_1\ldots p_k}:= \underset{\sigma\in\mathfrak{S}(k)}{\sum} T_{p_{\sigma(1)}\ldots p_{\sigma(k)}},\]
where \index{S@$\mathfrak{S}(k)$ symmetric group of $\{1,\ldots,k\}$|textbf} $\mathfrak{S}(k)$ is the symmetric group on the finite set $\{1,\ldots,k\}$ and $I=(p_1,\ldots,p_k)$ is a multi-index $k$-tuple. 
The two-by-two contraction of $\Sym T_I$, denoted by $\tr \Sym T_I$, is defined by
\index{Tr@$\tr\Sym T= g^{p_1p_2}\ldots g^{p_{k-1}p_k}\Sym T_{p_1\ldots p_k}$|textbf}
\begin{equation*}
\tr\Sym T_I:=\begin{cases} & g^{p_1p_2}\ldots g^{p_{k-1}p_k}\Sym T_{p_1\ldots p_k}, \textrm{ if $k$ is even},\\
&  g^{p_1p_2}\ldots g^{p_{k-2}p_{k-1}}\Sym T_{p_1\ldots p_k}, \textrm{ if $k$ is odd}.
\end{cases}
\end{equation*}
Note that when $k$ is odd, the two-by-two contraction of $\Sym T$ is a $1$-covariant tensor. 
We always use the Einstein summation convention, \emph{i.e.} each time an index occurs twice we sum over it.\\

Let $(M,g)$ be a compact Riemannian manifold of dimension $n$. We fix $x\in M$ and define: \index{o@$\omega=\inf \{k\in \mathbb N : "|\nabla^k Weyl_g(x)"|\neq0\}$|textbf} 
$\omega=\inf \{k\in \mathbb N : |\nabla^k Weyl_g(x)|\neq0\}$.  Lee and Parker \cite{LP}  proved that in each conformal class $[g]$, there exists a metric $g'$ which satisfies, in a geodesic normal chart, $\det(g')=1+O(r^N)$, with $N$  sufficiently large, where $r=d(x,\cdot)$. This result was extended by  Hebey and Vaugon \cite{HV} in the equivariant setting. If $G$ is a subgroup of the isometry group ${\rm Isom}(M,g)$ and $[g]^G$ denotes the $G$-invariant conformal class of $g$, then they proved
\begin{lemma}\label{lemme HV ex}
In each class  $[g]^G$, there exists a metric $g'$ which satisfies, in a geodesic normal chart, $\det(g')=1+O(r^N)$, with $N$  sufficiently large, where $r=d(x,\cdot)$. Moreover, in this chart, the Taylor expansion of $g'$ is given by
\begin{multline*}
 g'_{ij}(y)=\delta_{ij}+\hspace{-0.4cm}\sum_{\omega+4\leq m \leq 2\omega+6}\frac{2(m-3)}{(m-1)!}\nabla_{p_3\ldots p_{m-2}}R'_{ip_1p_2j}(x)x^{p_1}\ldots x^{p_{m-2}} +\\
K(\omega)\sum_{k=1}^n\nabla_{p_3\ldots p_{\omega+2}}R'_{ip_1p_2k}(x)\nabla_{p_{\omega+5}\ldots p_{2\omega+4}}R'_{jp_{\omega+3}p_{\omega+4}k}(x)x^{p_1}\ldots x^{p_{2\omega+4}}+O(r^{2\omega+5}),
\end{multline*} 
where $K(\omega)=\frac{(3\omega+8)(\omega+1)^2}{(2\omega+5)[(\omega+3)!]^2}$.
\end{lemma}

Using the above Taylor expansion of the metric and writing the metric as exponential of some symmetric matrix, Hebey and Vaugon proved the following formulas. 

\begin{theorem}[Hebey--Vaugon \cite{HV}] \label{HV-formulas}
Let $(M,g)$ be a compact Riemannian manifold of dimension $n$. Assume that $\det(g)=1+O(r^N)$, with $N$  sufficiently large, in a geodesic normal coordinate chart around $x$. Then the following statements hold 
\begin{enumerate}[label=\arabic*.]
\item For any nonnegative integer $ k\leq \omega-1,\; |\nabla^k\Riem(x)|=0$. \index{R@$\no\nabla^k\Riem\no$ $k$-th derivative norm of $\Riem$|textbf} 
\item For any nonnegative integer $k\leq 2\omega+1$, the $k$-th covariant derivatives of the Riemann, Ricci and scalar curvatures of $g'$ coincide with the usual partial derivatives and thus they commute. Namely, for any multi-index $\;\beta\in\{1,\ldots, n\}^k$, 
$$\nabla_\beta \Riem(x)=\partial_\beta \Riem(x),\;\nabla_\beta \Ric(x)=\partial_\beta \Ric(x),\; \nabla_\beta \scal_{g}(x)=\partial_\beta \scal_{g}(x).$$
\item   For any $\omega+2\leq m\leq 2\omega+3$, $\underset{p_1\ldots p_m}\Sym\nabla_{p_3\ldots p_m}\Ric_{p_{1}p_{2}}(x)=0$ and 
\begin{multline}\label{formula}
\underset{p_1\ldots p_{2\omega+4}}
\Sym    \biggl\{\nabla_{p_3\ldots p_{2\omega+4}}  \Ric_{p_{1}p_{2}}(x)\\
+ C(\omega)\nabla_{p_3\ldots p_{\omega+2}}R_{ip_1p_2j}(x)\cdot\nabla_{p_{\omega+5}\ldots p_{2\omega+4}}R_{ip_{\omega+3}p_{\omega+4}j}(x)\biggr\}=0,
\end{multline}
\end{enumerate}
where $C(\omega)=(\omega+1)^2(\omega+2)^2(2\omega+2)![(\omega+3)!]^{-2}$. \index{C@$C(\omega)=\frac{(\omega+1)^2(\omega+2)^2(2(\omega+1))"!}{[(\omega+3)"!]^{2}}$|textbf} 
\end{theorem}

Assume that $g$ satisfies the assumptions of Theorem \ref{HV-formulas}. 
It follows from the theorem that $\Delta^\ell\scal(x)=0$ and 
$|\nabla\Delta^\ell\scal(x)|=0$, for any nonnegative integer $\ell\leq \omega$.
Indeed, $\tr\underset{p_1\ldots p_{2\ell}}\Sym\nabla_{p_1\ldots p_{2\ell}}\scal(x)= 
(2\ell)!\Delta^\ell\scal(x)$ and 
$\underset{p_1\ldots p_{2\ell+3}}\Sym\nabla_{p_3\ldots p_{2\ell+3}}\Ric_{p_{1}p_{2}}(x)=
\underset{p_1\ldots p_{2\ell+2}}\Sym\bigl\{2\nabla_{p_2\ldots p_{2\ell+2}}\Ric_{p_{1}p_{2\ell+3}}(x)+(2\ell+1)\nabla_{p_3\ldots p_{2\ell+3}}\Ric_{p_{1}p_{2}}(x)\bigr\},$ since by Theorem~\ref{HV-formulas}, the covariant derivatives of order at most $2\omega+1$ of the Ricci curvature tensor and the scalar curvature commute. It follows that
\begin{gather}
2(2\ell)!(\ell+1)^2\Delta^\ell\scal(x)=\tr\underset{p_1\ldots p_{2\ell+2}}\Sym\nabla_{p_3\ldots p_{2\ell+2}}\Ric_{p_{1}p_{2}}(x)=0, \label{Delta scal1}\\
  (2\ell+2)!(\ell+2)\nabla_{p_{2\ell+3}}\Delta^\ell\scal(x)=\tr\underset{p_1\ldots p_{2\ell+3}}\Sym\nabla_{p_3\ldots p_{2\ell+3}}\Ric_{p_{1}p_{2}}(x)=~0.\label{Delta scal2}
\end{gather} 
Moreover, if we assume that $|\nabla^\omega\scal(x)|=0$, then the $(2\omega+2)$-covariant derivatives of the scalar curvature at $x$ commute (this fact will be proven at the end of the proof of Proposition~\ref{prop S}).    
For $r=d(x,\cdot)$ sufficiently  small, one can write the Taylor expansion of the scalar curvature in $B_r(x)$ and prove that there exists a positive constant $c(n,\omega)$  depending on $n$ and $\omega$, such that  $\frac{1}{vol (\partial B_r(x))}\int_{\partial B_r(x)}\scal \di\sigma=c(n,\omega) \sum_{\ell=0}^{\omega+1}\Delta^\ell\scal(x)r^{2\ell}+O(r^{2\omega+3})$.
 We conclude that  
 \begin{equation}\label{int vs delta}
 \frac{1}{vol (\partial B_r(x))}\int_{\partial B_r(x)}\scal \di\sigma=c(n,\omega)\Delta^{\omega+1}\scal(x)r^{2\omega+2}+O(r^{2\omega+3}).
 \end{equation}
Thus, $\int_{\partial B_r(x)}\scal \di\sigma$ and $\Delta^{\omega+1}\scal(x)$ have the same sign.


\section{Ricci symmetrization}\label{Ricci Sym}

\emph{From now on, $(M,g)$ denotes a compact Riemannian manifold of dimension $n$. We fix a point $x\in M$ and assume that $\det(g)=1+O(r^N)$, with $N$  sufficiently large, in a geodesic normal coordinate chart around $x$ (in particular, $g_{ij}(x)=\delta_{ij}$) and that all the covariant derivatives of order $\omega$ of the scalar curvature at $x$ vanish, i.e. $|\nabla^\omega\scal(x)|=0$. Since all the tensors are evaluated at $x$, we omit to mention $x$ for all tensors. For example, we write $\nabla_I\Ric_{ij}$ instead of $\nabla_I\Ric_{ij}(x)$}.

We introduce the following notation: 
\index{T@$\cT_\ell=\no\nabla^{\omega-2\ell}\Delta^\ell\Ric \no^2$|textbf} 
\index{R@$\cR_\ell= \no\nabla^{\omega-2\ell}\Delta^{\ell} \Riem \no^2$|textbf} 
\index{M@$\cM_\ell=\nabla_{K}\nabla_{a}\Delta^\ell\Ric_{bc}\cdot
\nabla_{K}\nabla_c\Delta^\ell\Ric_{ab}$|textbf} 
\index{N@$\cN_\ell=\nabla_{K'}\nabla_{cd}\Delta^\ell\Ric_{ab} \nabla_{K'}\nabla_{ab}\Delta^\ell\Ric_{cd}$|textbf} 
\begin{gather*}
\cR_\ell=\no\nabla^{\omega-2\ell}\Delta^\ell\Riem\no^2, \quad
\cT_\ell=|\nabla^{\omega-2\ell}\Delta^\ell\Ric|^2,\quad 0\leq  \ell \leq \frac{\omega}{2},\\
\cM_\ell=\nabla_{K}\nabla_{a}\Delta^\ell\Ric_{bc}\cdot
\nabla_{K}\nabla_c\Delta^\ell\Ric_{ab}, \quad 0\leq  \ell \leq \frac{\omega-1}{2},\\
\cN_\ell=\nabla_{K'}\nabla_{cd}\Delta^\ell\Ric_{ab} \cdot \nabla_{K'}\nabla_{ab}\Delta^\ell\Ric_{cd},\quad 0\leq  \ell \leq \frac{\omega-2}{2},
\end{gather*}
where we denote by $K$ and $K'$ two multi-indices $(\omega-2\ell-1)$ and $(\omega-2\ell-2)$-tuples respectively, and the fixed point $x$ is omitted.  By convention, $\cN_{\frac{\omega}{2}}
=\cM_{\frac{\omega}{2}}=0$, if $\omega$ is even and $\cN_{\frac{\omega-1}{2}}=0$, if $\omega$ is odd.

\begin{lemma}
For any $\ell\leq \frac{\omega}{2}$ the following equality holds 
 \begin{gather}
2 \cT_\ell +(\omega-2\ell)\{(\omega-2\ell-1) \cN_\ell+ 4\cM_\ell\}=0.\label{sym ric}
\end{gather}
\end{lemma}

\begin{proof}

By Theorem \ref{HV-formulas}, we have $\underset{p_1\ldots p_{\omega+2}}\Sym\nabla_{p_3\ldots p_{\omega+2}}\Ric_{p_1p_{2}}=0$. 
Since the Ricci tensor is symmetric and its covariant derivatives commute, 
one can also rewrite this identity as follows:
\[
\sum_{1\leq i<j\leq \omega+2}\nabla_{p_1\ldots \hat p_i\ldots \hat p_j\ldots p_{\omega+2}}\Ric_{p_{i}p_{j}}=0.
\]
After contracting by $g^{p_{k+1}p_{k+2}}\ldots g^{p_{\omega+1}p_{\omega+2}}$, 
with $k=\omega+2-2\ell\geq 2$ and $1\leq \ell\leq\frac{\omega}{2}$, we obtain
\begin{equation}\label{ric no}
 \sum_{1\leq i<j\leq k}\nabla_{p_1\ldots \hat p_i\ldots \hat p_j\ldots p_{k}}
\Delta^\ell\Ric_{p_{i}p_{j}}=0,
\end{equation}
since we assumed $|\nabla^\omega\scal|=0$.
Equality \eqref{sym ric} is obtained  by multiplying \eqref{ric no} with $\nabla_{p_1 \ldots p_{k-2}}\Delta^\ell\Ric_{p_{k-1}p_{k}}$.
\end{proof}

\begin{lemma}\label{lpos}
For any $\ell \leq \frac{\omega}{2}$, the following inequality holds
\begin{gather*}
2\cT_\ell+(\omega-2\ell)\cN_\ell\geq 0.
\end{gather*} 
Moreover, it follows that  $\cT_\ell+\cN_\ell-2\cM_\ell\geq 0$. 
\end{lemma}

\begin{proof} For $\omega$ even and $\ell=\frac{\omega}{2}$, by definition $\cM_\ell=\cN_\ell=0$. Thus, the inequalities holds. For  $\omega$ odd and $\ell=\frac{\omega-1}{2}$, by definition $\cN_\ell=0$ and by \eqref{sym ric}, we have $\cT_\ell+2\cM_\ell=0$. Hence $\cM_\ell$ is nonpositive, which proves the inequalities.

From now on, we assume that $\ell\leq \frac{\omega-2}{2} $ and set $k:=\omega-2\ell+2\geq 4$. 
We use the convention $\binom{p}{q}=0$ if $q>p$. Let $\Gamma(t)$ be the $k$-covariant tensor, depending on the real parameter $t$, whose components are given by
$$\Gamma(t)_{p_1\ldots p_k}=\nabla_{p_1 \ldots p_{k-2}}\Delta^\ell\Ric_{p_{k-1}p_{k}}t+
\sum_{1\leq i<j\leq k-2}\nabla_{p_1 \ldots \hat p_i\ldots \hat p_j\ldots p_{k}}\Delta^\ell\Ric_{p_{i}p_{j}}, $$
The square norm of $ \Gamma(t)$ is given by
\begin{multline*}
 \Gamma(t)_{p_1\ldots p_k}\Gamma(t)_{p_1\ldots p_k}=\cT_\ell t^2+\binom{k-2}{2}\cT_\ell+
2 \binom{k-2}{2}\cN_\ell t+\\
\binom{k-2}{2} \binom{k-4}{2} \cN_\ell+2 \binom{k-2}{2}(k-4)\cM_\ell,
\end{multline*}
which is, with respect to $t$, a second degree nonnegative polynomial. Thus its discriminant is nonpositive, namely 
$$(k-2)(k-3)(\cN_\ell)^2-\cT_\ell\{ 2\cT_\ell+(k-4)(k-5)\cN_\ell+ 
4(k-4)\cM_\ell\}\leq 0.$$ 
We substitute $\cM_\ell$, using \eqref{sym ric} and obtain
$$(k-2)(k-3)\biggl(\frac{\cN_\ell}{\cT_\ell}\biggr)^2+2(k-4)
\biggl(\frac{\cN_\ell}{\cT_\ell}
\biggr)-\frac{4}{k-2}\leq 0,$$
where we assumed that $\cT_\ell\neq 0$ (otherwise the proof of the lemma is trivial, since all terms involved vanish).
We conclude that $\frac{-2}{k-2}\leq \bigl(\frac{\cN_\ell}{\cT_\ell}\bigr)
\leq\frac{2}{(k-2)(k-3)}$. 
This proves the first inequality, since $\omega\geq 2\ell+2$. 
Substituting the value of $\cM_\ell$, given by \eqref{sym ric}, 
in $\cT_\ell+\cN_\ell-2\cM_\ell$ and using the first inequality of the lemma yields the second one. 
\end{proof}

\section{Proof of Aubin's theorem}
We recall that $(M,g)$ denotes a compact Riemannian manifold of dimension $n$ and $x\in M$ 
is a fixed point. We assumed that $\det(g)=1+O(r^N)$, with $N$  sufficiently large, 
in a geodesic normal coordinate chart around $x$ and that $|\nabla^\omega\scal(x)|=0$. 
We use the following notation for the symmetrizations that occur in our computation:
\begin{align*}
\cS & =  \tr \Sym \nabla_{K}\Ric_{ab},\quad & 
\cS_1 & =   \tr\Sym\nabla_{K}\scal,\\
\cS_2 & = \tr \Sym\nabla_{I} R_{iabj} \cdot \nabla_{J} \Ric_{ij},\quad & 
\cS_3 & =   \tr\Sym\nabla_{I}\Ric_{aj}\cdot\nabla_{J}\Ric_{bj},\\
\cS_4 &= \tr \Sym \nabla_{I}R_{iabj}\cdot\nabla_{J'}\nabla_i\Ric_{cj}, \quad &  \cQ(R)_{IJabcd} &=\Sym\nabla_IR_{iabj}\cdot\nabla_JR_{icdj},
\end{align*}
\index{S1@$\cS =  \tr \underset{\substack{K\cup \{a,b\}, \#K=2\omega+2}}{\Sym} \nabla_{K}\Ric_{ab}$|textbf}
\index{S11@$\cS_1=   \tr \underset{K,\#K=2\omega+2}{\Sym}\nabla_{K}\scal$|textbf}
\index{S2@$\cS_2=   \tr \underset{\substack{I\cup J\cup \{a,b\},\\ \#I=\#J=\omega}}{\Sym}\nabla_{I}\Ric_{aj}\cdot\nabla_{J}\Ric_{bj}$|textbf}
\index{S3@$\cS_3=   \tr \underset{\substack{I\cup J\cup \{a,b\},\\ \#I=\#J=\omega}}{\Sym}\nabla_{I}\Ric_{aj}\cdot\nabla_{J}\Ric_{bj}$|textbf}
\index{S4@$\cS_4= \tr\hspace{-0.3cm} \underset{\substack{I\cup J'\cup \{a,b,c\},\\ \#I=\#J'+1=\omega}}{\Sym} \nabla_{I}R_{iabj}\cdot\nabla_{J'}\nabla_i\Ric_{cj}$|textbf}
\index{Q@$\cQ(R)_{IJabcd}=\nabla_IR_{iabj}\nabla_JR_{icdj}$|textbf}
\noindent where $I$, $J$, $J'$ and $K$ denote multi-indices of length $\#K=2\omega+2$, $\#I=\#J=\omega$, $\#J'=\omega-1$ and the symmetrization is taken over all the indices which are not yet contracted. Note that $\cQ(R)$ is a symmetric $(2\omega+4)$-covariant tensor and $\cS$, $\cS_1,\ldots,\cS_4$ are real numbers. Using this notation, the two-by-two contraction of the Hebey--Vaugon formula~\eqref{formula} is $\cS+C(\omega) \tr \cQ(R)=0$.\\

Before we start the computation, let us here illustrate the idea of 
our proof of Theorem \ref{main thm}. By \eqref{int vs delta}, it is sufficient to show 
that $\Delta^{\omega+1}\scal<0$. The first step in the proof is to compute $\cS$, 
which is the two-by-two contraction of the covariant derivatives of the Ricci tensor. 
When the order of the covariant derivatives is less or equal to $2\omega+1$, we know 
that they commute (cf. Theorem \ref{HV-formulas}). Therefore, the two-by-two contraction 
of the covariant derivatives of the Ricci tensor is, up to a positive integer, equal to 
either $\Delta^\ell\scal$ or $\nabla\Delta^\ell\scal$ for some nonnegative integer $\ell$ 
(cf. \eqref{Delta scal1} and \eqref{Delta scal2}). However in $\cS$ and $\cS_1$, 
there are $2\omega+2$ covariant derivatives and in general they do not commute. 
Each time we commute two of them in $\cS$ or $\cS_1$, the Riemann curvature tensor occurs. 
By using the fact that for any nonnegative integer $k\leq \omega-1$, $|\nabla^k\Riem|=0$ 
and $|\nabla^\omega\scal|=0$,  we show that the $2\omega+2$ covariant derivatives of the 
scalar curvature in $\cS_1$ commute and thus $\cS_1=(2\omega+2)!\Delta^{\omega+1}\scal$. 
Using this fact, we prove that  $\cS$ is equal to the sum of a positive integer times 
$\Delta^{\omega+1}\scal$ and a positive combination, denoted $\cJ$, of the terms $\cS_2$, 
$\cS_3$ and $\cS_4$, which corresponds to the commutativity obstruction of the covariant 
derivatives in $\cS$. By taking into account the formula $\cS+C(\omega) \tr \cQ(R)=0$, 
it is sufficient to show that $\cJ+C(\omega) \tr \cQ(R)>0$. The second step is to compute 
the two-by-two contraction of $\cQ(R)$. We determine all possible terms that occur when 
contracting all the Riemann tensor entries of $\cQ(R)$ (i.e. in the definition of $\cQ(R)$, 
they correspond to $a$, $b$, $c$ and $d$). It turns out that $ \tr \cQ(R)$ is a positive integer combination of 27 terms (without using Bianchi identities), denoted $\cA_k$, where each $\cA_k$ is a two-by-two contraction of a symmetric tensor of the same form  as $\cQ(R)$, but with contracted Riemann tensor entries. Further, we prove that all  $\cA_k$'s are nonnegative and some of them are positive (they are summarized in Table \ref{table}). These 27 terms $\cA_k$ are themselves a positive combination of the following 4 nonnegative  terms $\{\cR_\ell, \cT_\ell,  \cT_{\ell}-2\cM_{\ell}+\cN_{\ell}, \cT_{\ell}-\cM_{\ell}\}$, which are defined in Section \ref{Ricci Sym}. The third step is to do the same thing as in the second step for $\cJ$. Namely, we compute all possible terms occurring in the two-by-two contraction. We establish that  $\cJ$ is also an integer (with positive and negative coefficients) combination  of the following 3 terms $\{\cT_\ell, \cM_{\ell}, \cN_{\ell}\}$. In the last step, one has to add the two combinations of $\tr\cQ(R)$ and $\cJ$ and check, using Lemma \ref{lpos} and \eqref{sym ric}, that each kind of term, occurring in  $\cJ+C(\omega) \tr \cQ(R)$, is nonnegative. We conclude by using the fact that $\cR_0$ is positive, which holds by the definition of $\omega$.\\

\begin{proposition}\label{prop S}
The following equality involving the above defined symmetrizations holds: 
\begin{equation}\label{equa symRic}
\cS=2(\omega+2)^2 (2\omega+2)!\Delta^{\omega+1}\scal+ C(\omega)\bigl\{2(\omega+3)^2(\cS_2+\cS_3)+2\omega(\omega+3)\cS_4\bigr\},
\end{equation}
where $C(\omega)=(\omega+1)^2(\omega+2)^2(2\omega+2)![(\omega+3)!]^{-2}$.
\end{proposition}
By Theorem \ref{HV-formulas}, we have $\cS+C(\omega) \tr \cQ(R)=0$,  If we substitute $\cS$ by 
its value given in Proposition \ref{prop S}, we obtain the following equality.

\begin{corollary}\label{equa cor}
 The following equality holds
\begin{equation}
-\Delta^{\omega+1}\scal = \frac{(\omega+1)^2}{2[(\omega+3)!]^{2}}
\biggl\{\tr \cQ(R)+2(\omega+3)\bigl[(\omega+3)(\cS_2+\cS_3)+\omega \cS_4\bigr]\biggr\}.
\end{equation}
\end{corollary}
\begin{remark}
Note that the equality above  is claimed by Aubin \cite[p. 279]{Aub3}, which in his notation, is written in the form $2(\omega + 2)^2 \tr \Sym \nabla_{\alpha\beta k l}R + C(\omega)I = 0$. 
\end{remark}

\begin{proof}[Proof of Proposition \ref{prop S}]
 In order to prove \eqref{equa symRic}, we start by decomposing the symmetrization of the 
$(2\omega+2)$-covariant derivatives 
of the Ricci tensor and then contract two-by-two as follows:
\begin{multline}\label{formgen}
\cS= 2(\omega+2)\biggl\{\cS_1+
\tr\hspace{-0.3cm} \sum_{\substack{\sigma\in\mathfrak{S}(2\omega+2)\\ 1\leq k\leq 2\omega+2}}\hspace{-0.3cm}\nabla_{p_{\sigma(1)} \ldots p_{\sigma(k-1)} q p_{\sigma(k)} \ldots p_{\sigma(2\omega+1)}}
\Ric_{p_{\sigma(2\omega+2)}q}\biggr\}\\
= 2(\omega+2)\biggl\{\cS_1+\tr\hspace{-0.7cm} \sum_{\substack{\sigma\in\mathfrak{S}(2\omega+2)\\ \omega+1\leq k\leq 2\omega+2}}\hspace{-0.5cm}(1+\omega\delta_{k(\omega+1)})
 \nabla_{p_{\sigma(1)}\ldots p_{\sigma(k-1)} q p_{\sigma(k)} \ldots p_{\sigma(2\omega+1)}}
\Ric_{p_{\sigma(2\omega+2)}q}\biggr\},
\end{multline}
where  $\delta_{ij}$ denotes the Kronecker delta. The last equality follows by using  the fact that for any $k\leq\omega-1$, $|\nabla^k\Riem|=0$ (cf. Theorem \ref{HV-formulas}), which implies that any two successive covariant derivatives of order at least $\omega+2$ and $\omega+3$ respectively, in the $(2\omega+2)$-covariant derivatives of the Ricci tensor, commute. Namely,  for  each $1\leq k\leq \omega+1$ the following equality holds
\begin{multline*}
\nabla_{p_{\sigma(1)}\ldots p_{\sigma(k-1)} q p_{\sigma(k)} \ldots p_{\sigma(2\omega+1)}}
\Ric_{p_{\sigma(2\omega+2)}q}=\nabla_{p_{\sigma(1)}\ldots p_{\sigma(\omega)} q p_{\sigma(\omega+1)} \ldots p_{\sigma(2\omega+1)}}
\Ric_{p_{\sigma(2\omega+2)}q}.
\end{multline*}

In the last sum over $k$ in \eqref{formgen}, we obtain for $k=2\omega+2$ a multiple of $\cS_1$, 
by applying the contracted second Bianchi identity  \eqref{bianchi}:

\begin{equation}\label{form1}
\tr \sum_{\sigma\in\mathfrak{S}(2\omega+2)}\nabla_{p_{\sigma(1)}\ldots p_{\sigma(2\omega+1)}q}\Ric_{p_{\sigma(2\omega+2)}q}=
\frac{1}{2}\cS_1.
\end{equation}

For all the other terms in the last sum over $k$ in \eqref{formgen} we move the index $q$ to the right in order to 
get it as last index of the derivatives and then apply again \eqref{form1}. 
Thus, for each $\omega+1\leq k\leq 2\omega+2$, we obtain a term equal to $\frac{1}{2}\cS_1$.
Summing up, we obtain that $\cS_1$ occurs in $\cS$ with coefficient 
$2(\omega+2)[1+\frac{1}{2}\sum_{k=\omega+1}^{2\omega+2}(1+\omega\delta_{k(\omega+1)})]=2(\omega+2)^2$.

In the terms of the last sum in \eqref{formgen}, each permutation of $q$ to the right gives rise to new terms involving the Riemannian curvature tensor. 
More precisely, for each $\omega+1\leq k\leq 2\omega+1$, the following formula holds: 

\begin{multline}\label{permute}
\nabla_{p_{\sigma(1)}\ldots p_{\sigma(k-1)}qp_{\sigma(k)}\ldots p_{\sigma(2\omega+1)}}
\Ric_{p_{\sigma(2\omega+2)}q}=\\
\nabla_{p_{\sigma(1)}\ldots p_{\sigma(k-1)}p_{\sigma(k)}qp_{\sigma(k+1)}\ldots p_{\sigma(2\omega+1)}}
\Ric_{p_{\sigma(2\omega+2)}q}+\\
\nabla_{p_{\sigma(1)}\ldots p_{\sigma(k-1)}}\biggl\{\sum_{j=k+1}^{2\omega+1}R_{pp_{\sigma(j)}p_{\sigma(k)}q}
\nabla_{p_{\sigma(k+1)}\ldots\hat p_{\sigma(j)}\ldots p_{\sigma(2\omega+1)}} \nabla_p \Ric_{p_{\sigma(2\omega+2)}q}+\\
R_{pp_{\sigma(2\omega+2)}p_{\sigma(k)}q}\nabla_{p_{\sigma(k+1)}\ldots p_{\sigma(2\omega+1)}}\Ric_{pq}+
\Ric_{p_{\sigma(k)}p}\nabla_{p_{\sigma(k+1)}\ldots p_{\sigma(2\omega+1)}}\Ric_{p_{\sigma(2\omega+2)}p}\biggr\},
\end{multline}
where we used the fact that the covariant derivatives commute up to order $2\omega+1$ (cf. Theorem \ref{HV-formulas}). Note that for $k=2\omega+1$, the sum over $j$ in \eqref{permute} is empty, meaning that the corresponding terms do not occur in this case.

In the first term of \eqref{permute}, we continue to move the index $q$ to the right, using repeatedly the same formula until $q$ attains the last position in the indices of the covariant derivatives.
For the remaining terms in \eqref{permute}, we first compute, using the Leibniz rule, their covariant derivatives. 
For example, for the second type of term, we have for each $\omega+1\leq k\leq 2\omega+1$:

\begin{multline*}
\nabla_{p_{\sigma(1)}\ldots p_{\sigma(k-1)}}(R_{pp_{\sigma(2\omega+2)}p_{\sigma(k)}q}\cdot
\nabla_{p_{\sigma(k+1)}\ldots p_{\sigma(2\omega+1)}}\Ric_{pq})=\\
\sum_{S\subseteq\{p_{\sigma(1)}, \ldots, p_{\sigma(k-1)}\}, \#S=\omega}\hspace{-1.1cm}
(\nabla_S R_{pp_{\sigma(2\omega+2)}p_{\sigma(k)}q})  
(\nabla_{\{p_{\sigma(1)} \ldots p_{\sigma(k-1)}\}\setminus S} 
\nabla_{p_{\sigma(k+1)}\ldots p_{\sigma(2\omega+1)}}\Ric_{pq}).
\end{multline*}

Taking now in the last equality the sum over all permutations $\sigma\in\mathfrak{S}(2\omega+2)$ and contracting 
two-by-two, we obtain for each $\omega+1 \leq k\leq 2\omega+1$:
\begin{multline*}
\tr \hspace{-0.5cm}\sum_{\sigma\in\mathfrak{S}(2\omega+2)}\hspace{-0.5cm}\nabla_{p_{\sigma(1)}\ldots,p_{\sigma(k-1)}}
(R_{pp_{\sigma(2\omega+2)}p_{\sigma(k)}q}
\nabla_{p_{\sigma(k+1)}\ldots p_{\sigma(2\omega+1)}}\Ric_{pq})= \binom{k-1}{\omega} \cS_2.
\end{multline*}

After iterating the formula \eqref{permute} sufficiently many times in order to transform all terms occurring in 
the last sum in \eqref{formgen} into terms having the index $q$ as the last index of the covariant derivates, 
we can compute as follows the coefficient of $\cS_2$ in~$\cS$:
\[2(\omega+2) \sum_{k=\omega+1}^{2\omega+1}\biggl\{(1+\omega\delta_{(\omega+1)k})\sum_{j=k-1}^{2\omega}\binom{j}{\omega}\biggr\}=
2(\omega+2)\sum_{k=\omega}^{2\omega}(k+1)\binom{k}{\omega}=2C(\omega)(\omega+3)^2.\]
The proof of the last equality is given in  Appendix  \ref{comb}.
Applying the same argument as above to the last type of term in \eqref{permute}, yields that $\cS_3$ occurs with the same multiplicity in $\cS$.  Now, we consider the remaining term in \eqref{permute}. For any $\omega+1\leq k\leq 2\omega$, we have

\begin{multline*}
\nabla_{p_{\sigma(1)}\ldots p_{\sigma(k-1)}}\biggl(\sum_{j=k+1}^{2\omega+1}R_{pp_{\sigma(j)}p_{\sigma(k)}q}
\nabla_{p_{\sigma(k+1)}\ldots\hat p_{\sigma(j)}\ldots p_{\sigma(2\omega+1)}} \nabla_p \Ric_{p_{\sigma(2\omega+2)}q}\biggr)=\\
\sum_{j=k+1}^{2\omega+1}\sum_{\substack{S\subseteq\{p_{\sigma(1)} \ldots p_{\sigma(k-1)}\}\\ \#S=\omega}}\hspace{-0.5cm}\nabla_S R_{pp_{\sigma(j)}p_{\sigma(k)}q} \nabla_{\{p_{\sigma(1)} \ldots p_{\sigma(2\omega+1)}\}\setminus (S\cup\{p_{\sigma(k)},p_{\sigma(j)}\})} \nabla_p \Ric_{p_{\sigma(2\omega+2)}q}.
\end{multline*}

Taking in this last equality the sum over all permutations $\sigma\in\mathfrak{S}(2\omega+2)$ and contracting 
two-by-two, we obtain for each $\omega+1 \leq k\leq 2\omega$:
\begin{multline}\label{term1derive}
\tr \hspace{-0.3cm}\sum_{\sigma\in\mathfrak{S}(2\omega+2)}\hspace{-0.4cm}\nabla_{p_{\sigma(1)}\ldots p_{\sigma(k-1)}}\sum_{j=k+1}^{2\omega+1}R_{pp_{\sigma(j)}p_{\sigma(k)}q}
\nabla_{p_{\sigma(k+1)}\ldots\hat p_{\sigma(j)}\ldots p_{\sigma(2\omega+1)}} \nabla_p \Ric_{p_{\sigma(2\omega+2)}q}=\\
 \sum_{j=k+1}^{2\omega+1}\binom{k-1}{\omega} \cS_4=(2\omega+1-k)\binom{k-1}{\omega} \cS_4
\end{multline}
As above, in order to find the coefficient of $\cS_4$ in $\cS$, we 
combine \eqref{formgen}, \eqref{term1derive} and we obtain
\begin{equation*}
2(\omega+2)\sum_{k=\omega+1}^{2\omega}\biggl\{(1+\omega\delta_{(\omega+1)k})\sum_{j=k-1}^{2\omega}(2\omega-j)\binom{j}{\omega}\biggr\}=2\omega(\omega+3)C(\omega),
\end{equation*}
The proof of the last equality is given in Appendix \ref{comb}. To finish the proof of the proposition, it remains to show that $\cS_1=(2\omega+2)!\Delta^{\omega+1}\scal$. In fact, when commuting any two successive covariant 
derivatives of the $(2\omega+2)$-covariant derivatives of the scalar curvature, there are curvature terms occurring. These are all of the form $\nabla^\alpha\Riem\nabla^\beta\scal$, with $\alpha+\beta=2\omega$. 
By Theorem \ref{HV-formulas}, we know that for each $k\leq\omega-1$, we have $|\nabla^k\Riem| =|\nabla^k\scal| =0$. Since we assumed that $|\nabla^\omega\scal|=0$, 
it follows that the $(2\omega+2)$-covariant 
derivatives of the scalar curvature commute.   
\end{proof}

\bigskip

We recall the definition of the symmetric $(2\omega+4)$-covariant tensor $\cQ(R)$, whose components $\cQ(R)_{p_1 \ldots p_{2\omega+4}}$ are given by: 
\begin{equation*}
\underset{\sigma\in\mathfrak{S}(2\omega+4)}{\sum} \hspace{-0.2cm}\nabla_{p_{\sigma(1)}\dots p_{\sigma(\omega)}} R_{i p_{\sigma(\omega+1)} p_{\sigma(\omega+2)} j}\cdot \nabla_{p_{\sigma(\omega+3)}\cdots p_{\sigma(2\omega+2)}} R_{i p_{\sigma(2\omega+3)} p_{\sigma(2\omega+4)} j}.
\end{equation*}
Our purpose is to show that $\tr\cQ(R)$ is a positive combination of nonnegative terms. In order to define these terms, we need first to introduce  some new notation.
For all nonnegative integers $\ell$ and $\beta$, satisfying $0\leq\beta\leq \omega$ and $0\leq\ell\leq\frac{\omega-\beta}{2}$, 
\index{T11@$\cT^{\omega-\beta-2\ell}_\ell=\tr \underset{I \cup J}{\Sym}  \nabla_{IK} \Delta^\ell \Ric_{ij}  \nabla_{JK}\Ric_{ij}$|textbf}
 \index{T12@$\cT^{\omega-2}_{1,1}=\tr \underset{I''\cup J''}{\Sym}\nabla_{I''}\Delta\Ric_{ij}\nabla_{J''}\Delta\Ric_{ij}$|textbf}
\begin{gather*}
\cT^{\omega-\beta-2\ell}_\ell:=\tr \underset{I \cup J}{\Sym}  \nabla_{I}  \nabla_K \Delta^\ell \Ric_{ij}  \nabla_{J}\nabla_K\Ric_{ij},\\
\cT^{\omega-2}_{1,1}:=\tr \underset{I''\cup J''}{\Sym}\nabla_{I''}\Delta\Ric_{ij}\nabla_{J''}\Delta\Ric_{ij},
\end{gather*}
where here $I$, $J$,  $I''$, $J''$ and $K$ are multi-indices of cardinality $\omega-\beta-2\ell$, $\omega-\beta$, $\omega-2$, $\omega-2$ and  $\beta$ respectively. In other words, $\cT^{\omega-\beta-2\ell}_\ell$ is obtained by taking the product of $(\omega-2\ell)$-covariant derivatives of $\Delta^\ell\Ric$ with $\omega$-covariant derivatives of $\Ric$, where $\beta$ covariant derivatives of both factors are already contracted to each other and by contracting two-by-two the symmetrization over the remaining covariant derivatives.\\
Set \index{T121@$\cT^{\omega-\beta}=\cT^{\omega-\beta}_0$|textbf} $\cT^{\omega-\beta}=\cT^{\omega-\beta}_0$. Note that \index{T13@$\cT^{0}_\ell=(2\ell)\fac\cT_\ell$|textbf} $\cT^{0}_\ell=(2\ell)!\cT_\ell$ (cf. Section \ref{Ricci Sym} for the definition of $\cT_\ell$).
\index{M11@$\cM^{\omega-1-\beta-2\ell}_\ell=\tr \underset{I' \cup J'}{\Sym}  \nabla_{I'Kp} \Delta^\ell \Ric_{ij}  \nabla_{J'Ki}\Ric_{pj}$|textbf}
\index{N11@$\cN^{\omega-2-\beta-2\ell}_\ell=\tr \underset{I'' \cup J''}{\Sym}  \nabla_{I''Kpq} \Delta^\ell \Ric_{ij}  \nabla_{J''Kij}\Ric_{pq}$|textbf}
The terms $\cR^{\omega-\beta-2\ell}_\ell$, $\cM^{\omega-\beta-2\ell}_\ell$, $\cN^{\omega-\beta-2\ell}_\ell$, $\cR^{\omega-\beta}$, $\cM^{\omega-\beta}$ and $\cN^{\omega-\beta}$ are defined in the same way as $\cT^{\omega-\beta-2\ell}_\ell$ and $\cT^{\omega-\beta}$. Now we show that these terms are integer combinations of $\cR_\ell$, $\cT_\ell$,  $\cM_\ell$ and $\cN_\ell$. First, we note that 
\begin{multline*}
\cT^{\omega-\beta}=\tr \Sym  \nabla_{I}  \nabla_K \Ric_{ij}  \nabla_{J}\nabla_K\Ric_{ij}\\
 =2(\omega-\beta)\tr\Sym \{ (\omega-\beta) \nabla_{I'Ka} \Ric_{ij}  
\nabla_{JKa}\Ric_{ij} + (\omega-\beta-1)\nabla_{I''Kaa} \Ric_{ij}  
\nabla_{JK}\Ric_{ij} \}\\
 =2(\omega-\beta)\{(\omega-\beta)\cT^{\omega-\beta-1}+ 
(\omega-\beta-1)\cT^{\omega-\beta-2}_1\}.
\end{multline*} 
It follows that for any $2\leq\gamma\leq \omega$, we have 
\begin{equation}\label{ind T}
\cT^\gamma=2\gamma(\gamma-1)\cT^{\gamma-2}_1+2\gamma^2\cT^{\gamma-1}.
\end{equation}
Similarly, we obtain the following formulas:
\begin{gather}
\cT^{\gamma-2}_1=2\gamma(\gamma-1)\cT^{\gamma-3}_1+2(\gamma-3)(\gamma-1)\cT^{\gamma-4}_2,\label{ind T_1}\\
\cT^{\gamma-2}_1=2(\gamma-1)(\gamma-2)\cT^{\gamma-3}_1+2(\gamma-1)^2\cT^{\gamma-2}_{1,1}.\label{ind T_11}
\end{gather} 
Using \eqref{ind T}, \eqref{ind T_1} and \eqref{ind T_11}, we prove, by induction on $\gamma$, the following formulas
\begin{gather}\label{T ind}
\cT^{\gamma}=\sum_{\ell=0}^{[\frac{\gamma}{2}]}d^\gamma_\ell \cT_\ell,\; \text{ for all } 0\leq\gamma\leq\omega; \quad \cT_1^{\gamma-2}=\sum_{\ell=1}^{[\frac{\gamma}{2}]}e^{\gamma}_{\ell-1} \cT_\ell, \text{ for all } 2\leq\gamma\leq\omega,\\
\cT_{1,1}^{\gamma-2}=\sum_{\ell=1}^{[\frac{\gamma}{2}]}d^{\gamma-2}_{\ell-1} \cT_\ell, \text{ for all } 2\leq\gamma\leq\omega \label{T1,1},
\end{gather}
where $d^{\gamma}_{\ell}:=\frac{ 2^{\gamma-2\ell}(\gamma!)^3}{(\gamma-2\ell)!(\ell!)^2}$ and 
$e^\gamma_\ell:=\frac{ 2^{\gamma-2\ell-2}\gamma !(\gamma-1)!(\gamma-2)!}{(\gamma-2\ell-2)!(\ell+1) !\ell !}$.
\index{d@$d^{\gamma}_{\ell}=\frac{ 2^{\gamma-2\ell}(\gamma "!)^3}{(\gamma-2\ell)"!(\ell "!)^2}$|textbf}   
\index{e@$e^{\gamma}_{\ell}=\frac{2^{\gamma-2\ell-2}\gamma "!(\gamma-1)"!(\gamma-2)"! }{(\gamma-2\ell-2)"!(\ell+1)"!\ell "!}$|textbf}
Note that the equalities \eqref{ind T} and \eqref{T ind} hold also for $\cR^\gamma$, $\cM^{\gamma-1}$ and $\cN^{\gamma-2}$.
\begin{claim}\label{claim}
$\tr\cQ(R)=\sum_{k=1}^{27}u_k\cA_k$ and $\cA_k=\cB_k$ for any $1\leq k\leq 27$, where the terms $\cA_k$, $\cB_k$ and the coefficients $u_k$ are defined in Table \ref{table}. 
\end{claim}

\begin{table}[htbp] 
\begin{center}                
\begin{tabular}{|c|l|l|l|}
\hline
$k$   &   $\cA_k$ &  $u_k$&  $\cB_k$  \\
 \hline
\hline
1        & $\tr {\Sym}\nabla_I\Ric_{ij}\nabla_J\Ric_{ij}$  &  $c_1$  & 
$\cT^{\omega}$  \\
\hline
2     & $-\tr {\Sym}\nabla_{I'c}\Ric_{ij}\nabla_{J'b}R_{ibcj}$ & $2c_2$   
& $\cT^{\omega-1}-\cM^{\omega-1}$ \\
\hline
3    & $-\tr {\Sym}\nabla_{I}\Ric_{ij}\nabla_{J''bc}R_{ibcj}$ &  $2c_3$ 
&   $\cT^{\omega-2}_{1}$  \\
\hline
4    & $-\tr {\Sym}\nabla_{I'b}\Ric_{ij}\nabla_{J'c}R_{ibcj}$ & $2c_2$   
 & $\cB_2$ \\
\hline
5     & $-\tr {\Sym}\nabla_{I''bc}\Ric_{ij}\nabla_{J}R_{ibcj}$ & $2c_3$  
&  $\sum_{\ell=0}^{[\frac{\omega-2}{2}]}e^{\omega}_\ell(\cT_\ell-2\cM_\ell+
\cN_\ell)$  \\
\hline
6     & $\tr {\Sym}\nabla_{I}R_{iabj}\nabla_{J}R_{iabj}$ &  $c_1$ 
&  $\cR^{\omega}$ \\
\hline
7   & $\tr {\Sym}\nabla_{I'b}R_{iabj}\nabla_{J'c}R_{iacj}$ & $c_2$  
&  $2  \cB_2$   \\
\hline
8    & $\tr {\Sym}\nabla_{I''cb}R_{iabj}\nabla_{J}R_{iacj}$ & $2c_3$ 
 & $2\cB_5$  \\
\hline
9     & $\tr {\Sym}\nabla_{I'c}R_{iabj}\nabla_{J'b}R_{iacj}$ & $c_2$ 
 &   $\frac{1}{2} \cR^{\omega-1}$   \\
\hline
10    & $\tr {\Sym}\nabla_{I}R_{iabj}\nabla_{J}R_{ibaj}$ &  $c_1$ 
& $\frac{1}{2} \cB_6$   \\
\hline
11    & $\tr {\Sym}\nabla_{I'b}R_{iabj}\nabla_{J'c}R_{icaj}$ & $2c_2$  
  & $\cB_2$  \\
\hline
12    & $\tr {\Sym}\nabla_{I''bc}R_{iabj}\nabla_{J}R_{icaj}$ & $2c_3$  
 & $\cB_5$  \\
\hline
13   & $\tr {\Sym}\nabla_{I'c}R_{iabj}\nabla_{J'b}R_{icaj}$ &   $2c_2$
 &  $\frac{1}{2}\cB_9$  \\
\hline
14    & $\tr {\Sym}\nabla_{I}R_{iabj}\nabla_{J''bc}R_{icaj}$ &   $2c_3$
 & $\cB_5$  \\
\hline
15    & $\tr {\Sym}\nabla_{I}R_{iabj}\nabla_{J''ac}R_{icbj}$ &  $2c_3$  
& $2\cB_5$   \\
\hline
16   & $\tr {\Sym}\nabla_{I'c}R_{iabj}\nabla_{J'a}R_{icbj}$ &   $c_2$   
& $\cB_9$  \\
\hline
17    & $\tr {\Sym}\nabla_{I'a}R_{iabj}\nabla_{J'c}R_{icbj}$ &    $c_2$  
& $2\cB_2$  \\
\hline
18     & $\tr {\Sym}\nabla_{I''ab}R_{iabj}\nabla_{J''cd}R_{icdj}$ & $c_4$  
  & $\cT^{\omega-2}_{1,1}$ \\
\hline
19  & $\tr {\Sym}\nabla_{I''''abcd}R_{iabj}\nabla_{J}R_{icdj}$ &   $2c_5$ 
&  $ \sum_{\ell=0}^{[\frac{\omega-2}{2}]}\frac{\ell e^{\omega}_\ell (\cT_\ell-2\cM_\ell+
\cN_\ell)}{(\omega-1)(\omega-2)(\omega-3)}$ \\
\hline
20    & $\tr {\Sym}\nabla_{I'''abc}R_{iabj}\nabla_{J'd}R_{icdj}$ &   $2c_6$ 
 &  $\cT^{\omega-3}_{1}-\cM^{\omega-3}_{1}$\\
\hline
21   & $\tr {\Sym}\nabla_{I'''abd}R_{iabj}\nabla_{J'c}R_{icdj}$ &    $2c_6$
 &  $\cB_{20}$  
\\
\hline
22   & $\tr {\Sym}\nabla_{I'c}R_{iabj}\nabla_{J'''dab}R_{icdj}$  &  $2c_6$ 
  & $\sum_{\ell=0}^{[\frac{\omega-3}{2}]}e^{\omega-1}_\ell(\cT_\ell-2\cM_\ell+
\cN_\ell)$ \\
\hline
23    & $\tr {\Sym}\nabla_{I'd}R_{iabj}\nabla_{J'''cab}R_{icdj}$ & $2c_6$ 
 &   $\cB_{22}$  \\
\hline
24    & $\tr {\Sym}\nabla_{I''cd}R_{iabj}\nabla_{J''ab}R_{icdj}$ &   $c_4$
 & $\frac{1}{2}\cR^{\omega-2}$\\
\hline
25    & $\tr {\Sym}\nabla_{I''da}R_{iabj}\nabla_{J''bc}R_{icdj})$ &  $c_4$
 & $\cT^{\omega-2}-\cM^{\omega-2}$ \\
\hline
26     & $\tr {\Sym}\nabla_{I''ca}R_{iabj}\nabla_{J''bd}R_{icdj}$  &  $2c_4$
 &  $\cT^{\omega-2}-2\cM^{\omega-2}+\cN^{\omega-2}$ \\
\hline
27    & $\tr {\Sym}\nabla_{I''cb}R_{iabj}\nabla_{J''ad}R_{icdj}$  &  $c_4$ 
 &  $\cB_{25}$   \\
\hline
\end{tabular}
\end{center}
\vspace{0.4cm}
\caption{}\label{table}
\end{table}

Each term $\cA_k$ \index{A@$\cA_k$|textbf} of the second column denotes a  two-by-two contraction of a symmetric tensor, where the symmetrization is taken over all the indices which are not already contracted. In the fourth column, we define the terms $\cB_k$.  \index{B@$\cB_k$|textbf}  We prove using the two Bianchi identities that $\cA_k=\cB_k$ (cf. Appendix \ref{A=B}). The coefficients $u_k$ are given in terms of the following positive integers: 
\index{c_k@$c_1,\ldots,c_6$|textbf}
\begin{align*}
c_1&:=(2\omega+4)(2\omega+2),  &  c_3  & :=2\omega^2(\omega-1)c_1, & c_5 := & 4\omega^2(\omega-1)^2(\omega-2)(\omega-3)c_1,\\
c_2 & :=  2\omega^3c_1, & c_4&:=4\omega^3(\omega-1)^3c_1,   &  c_6  := & 4\omega^3(\omega-1)^2(\omega-2)c_1.
\end{align*}
Let us recall that the covariant derivates of the Riemann curvature tensor already commute (\emph{cf.} Theorem \ref{HV-formulas}).
This allows us to determine all types of terms that are obtained after taking the two-by-two contractions of $\cQ(R)$ involving the indices that occur directly in the Riemann curvature tensor, namely of the following kind:
\begin{equation}\label{type}
g^{a\alpha}g^{b\beta}g^{c\gamma}g^{d\delta}\nabla_I R_{iabj}\cdot\nabla_J R_{icdj},
\end{equation}
where $\{a,b,c,d\}\cup I \cup J=\{p_1, \ldots, p_{2\omega+4}\}$ is an arbitrary partition of $\{p_1, \ldots, p_{2\omega+4}\}$ into subsets of cardinality $4$, respectively \index{IJ@$I$, $J$ multi-indices sets $\#I=\#J=\omega$|textbf} $\#I=\#J=\omega$ and $\{\alpha,\beta,\gamma,\delta\}$ is any subset of $4$ indices of $\{p_1, \ldots, p_{2\omega+4}\}$, such that $\alpha\neq a, \beta\neq b, \gamma\neq c$ and $\delta\neq d$.

After taking into account the commutativity of the product, we identify $27$ types 
of terms like \eqref{type},  which are denoted by $\cA_k$ and listed  in the second column of Table \ref{table}. Moreover, if $I$ is a multi-index of length $\omega$, then 
$I'$ \index{IJ@$I'$, $I''$ multi-indices sets $\#I'+1=\#I''+2=\omega$|textbf}
 denotes a multi-index of length $\omega-1$, $I''$ a multi-index of length 
$\omega-2$ and so on.  
We also note that each of the sums over the contraction indices $a,b,c,d$ 
runs from $1$ to $n$ and that not all $27$ types occur for $\omega\leq 3$, 
for instance for $\omega=2$ the terms $\cA_{19},\ldots, \cA_{23}$ in the table do 
not occur. 

We further compute the multiplicity of each such term of 
type \eqref{type} in the two-by-two contraction of $\cQ(R)$. 
We write it down in the third column of Table \ref{table}.

Let us here only illustrate how we determined the multiplicity for one of the terms. 
For example, the third  term $\cA_3=\tr \Sym\nabla_I R_{iaaj}\cdot\nabla_{J''bc} R_{ibcj}$ has 
multiplicity equal to 
$2c_3$ as follows. 
For the choice of the three indices $a$, $b$ and $c$ among the $\omega+2$ 
pairs of distinct indices over which the contraction is done we have $\omega(\omega+1)(\omega+2)$ 
possibilities, and for each such fixed choice we may permute the two indices of each 
such pair, thus multiplying the coefficient by $2^3$. Further we have to multiply by 
$\omega(\omega-1)$, since we may choose for $b$ and $c$ any two indices (whose order also 
counts) of the multi-index of length $\omega$ corresponding to the covariant 
derivatives, because these commute, as remarked above. 
The remaining $2\omega-2$ indices in $I$ and $J''$ may be permuted  arbitrarily. 
One still has to take into account that the two factors of each term commute and 
since the factors of this term are not symmetric, its coefficient has to be 
doubled. Concluding, we obtain the claimed multiplicity. For all the other types of terms in the table the multiplicity is computed similarly. 

To verify that we obtained all the possible terms, we check, by summing up the multiplicities, that
\begin{equation*}
(2\omega)!(u_1+u_6+u_{10})+(2\omega-2)!\underset{\substack{k=2 \\ k\notin \{6,10\}}}{\sum^{17}}u_k+(2\omega-4)!\sum_{k=18}^{27}u_k=(2\omega+4)!.
\end{equation*}

Using the two Bianchi identities, we show that for any $1\leq k\leq 27$,  $\cA_k=\cB_k$ (the details of the computation can be found in  Appendix \ref{A=B}). This finishes the proof of the claim. As a consequence, we obtain the first part of the following result. 
\begin{proposition}\label{Q(R)>0}
The following equality holds: $\tr \cQ(R)=\sum_{k=1}^{27}u_k\cB_k$, where the $u_k$'s are positive integer coefficients and $\cB_k$'s are nonnegative terms,  defined in Table \ref{table}. Furthermore, it follows that
\begin{multline*}
\tr \cQ(R)>(2\omega+4)(2\omega+2)\bigl\{ 
\cT^{\omega}+4\omega^2(\omega-1)\cT^{\omega-2}_1+4\omega^3(\omega-1)^3\cT_{1,1}^{\omega-2}+\\
4\omega^3(\omega+4)(\cT^{\omega-1}-\cM^{\omega-1})+8\omega^3(\omega-1)^2(\omega-2)(\cT^{\omega-3}_1-\cM^{\omega-3}_1)+
4\omega^2(\omega-1)(\omega+7)\cB_{5}\bigr\}.
\end{multline*}
\end{proposition}

\begin{proof}

We still have to show that the $\cB_k$'s are nonnegative. For this, we further contract two-by-two the remaining indices occurring in the covariant derivatives in each of the $\cB_k$'s. 
Since the covariant derivatives of order less or equal to $\omega$ occurring in the terms $\cB_k$  commute, one establishes that for  any $1\leq k\leq 27$, $\cB_k=\sum_{\ell=0}^{[\frac{\omega}{2}]}\alpha_{k\ell} \cB_{k\ell}$, where $\alpha_{k\ell}$ are nonnegative integers representing the multiplicities and 
$$\cB_{k\ell}\in \{\cR_\ell, \cT_\ell,  \cT_{\ell}-2\cM_{\ell}+\cN_{\ell}, \cT_{\ell}-\cM_{\ell}\},$$
due to equalities \eqref{T ind}, \eqref{T1,1}, which hold for $\cR^\omega$, $\cT^\omega$, $\cM^{\omega-1}$ and  $\cN^{\omega-2}$.
Clearly, $\cR_\ell$ and $\cT_\ell$ are nonnegative. By Lemma~\ref{lpos},  $\cT_{\ell}-2\cM_{\ell}+\cN_{\ell}$ are nonnegative. The terms $\cT_{\ell}-\cM_{\ell}$ are 
also nonnegative, as it follows by writing down that the square norm of the $k$-covariant tensor 
$\nabla_{Kc}\Delta^\ell\Ric_{ab}-
\nabla_{Ka}\Delta^\ell\Ric_{bc}$ is nonnegative, where $K$ is a multi-index $(\omega-2\ell-1)$-tuple. 
Therefore the $\cB_k$'s are nonnegative.

Substituting in  the equality $\tr \cQ(R)=\sum_{k=1}^{27}u_k\cB_k$, proven in Claim \ref{claim}, the values of $\cB_k$ and $u_k$ from Table \ref{table}, we obtain
\begin{multline}\label{Q(R)=}
\tr \cQ(R)=(2\omega+4)(2\omega+2)\biggl\{ \frac{3}{2}\cR^{\omega}+3\omega^3\cR^{\omega-1}+2\omega^3(\omega-1)^3\cR^{\omega-2}+\\
\cT^{\omega}+4\omega^2(\omega-1)\cT^{\omega-2}_1+4\omega^3(\omega-1)^3\cT_{1,1}^{\omega-2}+
20\omega^3(\cT^{\omega-1}-\cM^{\omega-1})+\\
8\omega^3(\omega-1)^3(\cT^{\omega-2}-\cM^{\omega-2})+16\omega^3(\omega-1)^2(\omega-2)(\cT^{\omega-3}_1-\cM^{\omega-3}_1)+\\
28\omega^2(\omega-1)\cB_{5}+8\omega^2(\omega-1)^2\bigl[(\omega-2)(\omega-3)\cB_{19}+2\omega(\omega-2)\cB_{22}+\omega(\omega-1)\cB_{26}\bigr]\biggr\}.
\end{multline}

It order to prove the inequality stated in the proposition, we use the following equalities:

\begin{gather}
\cT^{\omega-2}-\cM^{\omega-2}=\frac{1}{2(\omega-1)^2}(\cT^{\omega-1}-\cM^{\omega-1})- 
\frac{\omega-2}{\omega-1}(\cT^{\omega-3}_1-\cM^{\omega-3}_1),\label{t-m}\\
\cB_{26}=\frac{1}{2(\omega-1)^2}\cB_{5}-\frac{\omega-2}{\omega-1}\cB_{22}\label{b=b},
\end{gather}
where the first one holds by \eqref{ind T} and the second one holds from the definition of $\cB_{5}$, $\cB_{22}$ and $\cB_{26}$. 
Substituting \eqref{t-m} and \eqref{b=b} in \eqref{Q(R)=} and using the fact that by the definition of $\omega$, $\cR_0:=|\nabla^\omega\Riem|^2$ is positive, we  obtain the desired strict inequality. 
\end{proof}

\begin{remark}
There is another method to prove the equality $\tr \cQ(R)=\sum_{k=1}^{27}u_k\cB_k$ of Proposition \ref{Q(R)>0}. One may contract the 4 entries of the Riemann tensor in $\cQ(R)$ one by one and then use the fact that $\cA_k=\cB_k$. This method is used in the proof of Lemma \ref{S lemma} for the computation of the $\cS_k$'s, with $2\leq k\leq4$.  
\end{remark}


\begin{lemma}\label{S lemma}
The following identities hold:
\begin{gather*}
 \cS_3=2(\omega+1)(\cT^{\omega}+2\omega^3\cM^{\omega-1}),\\
-\cS_2=2(\omega+1)\{\cT^\omega+4\omega^3(\cT^{\omega-1}-\cM^{\omega-1})+ 2\omega^2(\omega-1)\cT^{\omega-2}_1+ 2\omega^2(\omega-1)\cB_{5}\},\\
\cS_4 = 4\omega(\omega+1)\bigl\{\omega\cT^{\omega-1}-\omega(\omega+1)\cM^{\omega-1}+(\omega-1)\cB_{5}
+2\omega(\omega-1)^3\cN^{\omega-2}\bigr\}.
\end{gather*} 
\end{lemma}

\begin{proof}
Recall that the $\cS_k$'s are defined in the beginning of this section. First, we contract the index $b$, occurring in the symmetric tensor which defines $\cS_2$, with all the other indices not yet contracted (i.e. the indices in $I\cup J\cup\{a\}$), we obtain: 
\begin{equation*}
\hspace{-0.3cm}\frac{-\cS_2}{2\omega+2}= 
\tr \Sym \bigl\{\nabla_{ I } \Ric_{i j} \nabla_{J}\Ric_{ij}- \omega\nabla_{I'} \nabla_b R_{i abj} \nabla_{J}\Ric_{ij}-
\omega\nabla_{I} R_{i abj} \nabla_{J'}\nabla_b\Ric_{ij}\bigr\},
\end{equation*} 
 We continue, by contracting the index $a$ with all possible indices and we obtain:
 \begin{equation*}
 \begin{split}
\hspace{-0.3cm}\frac{-\cS_2}{2\omega+2}
=  &\cT^\omega-2\omega^2(\omega-1)\tr \Sym \biggl\{\nabla_{I''} \nabla_{ab} R_{i abj} \nabla_{J}\Ric_{ij}+
 \nabla_{I}R_{i abj} \nabla_{J''} \nabla_{ab} \Ric_{ij}\biggr\}-\\
& 2\omega^3\tr \Sym \biggl\{\nabla_{I'} \nabla_{b} R_{i abj} \nabla_{J'} \nabla_{a} \Ric_{ij}+  \nabla_{I'} \nabla_{a} R_{i abj}\nabla_{J'} \nabla_{b} \Ric_{ij}\biggr\}\\
=& \frac{1}{c_1}\{c_1\cA_{1}+c_2\cA_{2}+c_3\cA_{3}+c_2\cA_{4}+c_3\cA_{5}\},
\end{split}
\end{equation*}
where we used the notation of Table \ref{table}. Using the fact that $\cA_k=\cB_k$, we obtain the claimed equality for $\cS_2$.
Using the same method, which consists of contracting the entries of the Riemann and Ricci curvature tensors, we compute $\cS_3$ and $\cS_4$. 

The identity holds for $\cS_3$,  using the contracted second Bianchi identity and the fact that $|\nabla^\omega\scal|=0$. For $\cS_4$, we first contract the index $a$ with all the other not yet contracted indices and obtain:
\begin{multline*}
\frac{\cS_4}{2\omega+2}
=  \tr \Sym  \bigl\{-\nabla_{I} \Ric_{ij} \nabla_{J'i} \Ric_{cj}+\omega \nabla_{I'a} R_{iabj} \nabla_{J'i} \Ric_{cj}+\\
 (\omega-1) \nabla_{I} R_{iabj} \nabla_{J''ia} \Ric_{cj}+ \nabla_{I} R_{iabj} \nabla_{J'i} \Ric_{aj} \bigr\}.  
\end{multline*}
The third term of the right hand side in the last equality vanishes since the Riemann tensor is skew-symmetric with respect to the two first entries and the covariant derivatives of the Ricci tensor commute. By the second Bianchi identity, we have $\nabla_{a} R_{iabj}=\nabla_{j}\Ric_{ib}-\nabla_{b}\Ric_{ij}$. 
Substituting in the second term and using the fact that $|\nabla^\omega\scal|=0$, we obtain
\begin{multline*}
\frac{\cS_4}{2\omega+2}
=  \bigl\{-2\omega^2\cM^{\omega-1}+2\omega^2(\omega-1)[2(\omega-1)^2\cN^{\omega-2}- \cM^{\omega-1}]+
\\ 2\omega(\omega-1)\cA_5+2\omega^2 \cA_4\bigr\}.
\end{multline*}
Therefore, the equality corresponding to $\cS_4$ holds, since $\cA_4=\cB_4$ and $\cA_5=\cB_5$ (cf. Table \ref{table}).

\end{proof}

\begin{proof}[End of Theorem \ref{main thm} proof]  
By Corollary \ref{equa cor}, it is sufficient to prove that \index{I@$\cI$|textbf}
\begin{equation*}
\cI:= \tr \cQ(R)+2(\omega+3)[(\omega+3)(\cS_2+\cS_3)+\omega \cS_4]>0.
\end{equation*}
By Lemma \ref{S lemma}, it follows that 

\begin{multline}\label{sum S}
2(\omega+3)[(\omega+3)(\cS_2+\cS_3)+\omega \cS_4]= 
4(\omega+3)(\omega+1)\omega^2\biggl\{-2\omega(2\omega+5)(\cT^{\omega-1}-\cM^{\omega-1})\\
-2(\omega-1)(\omega+3)\cT_1^{\omega-2}
-2(\omega-1)(\omega+2)\cB_{5}+6\omega \cM^{\omega-1}+
4\omega(\omega-1)^3\cN^{\omega-2}\biggr\}.
\end{multline}
Therefore, by Proposition \ref{Q(R)>0} and \eqref{sum S}, we obtain
\begin{multline*}
\frac{\cI}{4(\omega+1)}>2\omega^3(\omega+1)(\cT^{\omega-1}-\cM^{\omega-1})+
(\omega+2)
\bigl\{\cT^{\omega}+8\omega^3(\omega-1)^2(\omega-2)(\cT^{\omega-3}_1-\cM^{\omega-3}_1)\bigr\}+\\
2\omega^2(\omega-1)\bigl\{2\omega(\omega+2)(\omega-1)^2\cT_{1,1}^{\omega-2}-(\omega^2+4\omega+5)\cT^{\omega-2}_1\bigr\}+\\
(\omega+3)\bigl\{6\omega^3 \cM^{\omega-1}+4\omega^3(\omega-1)^3
\cN^{\omega-2}\bigr\},
\end{multline*}
where we used the fact that $\cB_5$ is nonnegative. 
Set 

\index{I@$\cI_1$, $\cI_2$|textbf}
\begin{multline*}
\cI_1 :=  2\omega^2(\omega-1)\bigl\{2\omega(\omega+2)(\omega-1)^2
\cT_{1,1}^{\omega-2}-(\omega^2+4\omega +5)\cT^{\omega-2}_1+\\
 4\omega(\omega+2)(\omega-1)(\omega-2)(\cT^{\omega-3}_1-\cM^{\omega-3}_1)\bigr\},
\end{multline*}
\begin{multline*}
\cI_2 :=  2\omega^3(\omega+1)(\cT^{\omega-1}-\cM^{\omega-1})+ 
(\omega+2)\cT^{\omega}+\\
  (\omega+3)\{6\omega^3 \cM^{\omega-1}+ 4\omega^3(\omega-1)^3\cN^{\omega-2}\}.
\end{multline*}
In this notation, we have $\cI>4(\omega+1)(\cI_1+\cI_2)$. 
In order to finish the proof, we compute $\cI_1$ and $\cI_2$ and  
show that $\cI_1+\cI_2\geq 0$. For the computation of $\cI_2$, we have

\begin{equation*}
\begin{split}
   6\omega^3 \cM^{\omega-1}+4\omega^3(\omega-1)^3\cN^{\omega-2} &=6\omega^3\sum_{\ell=0}^{[\frac{\omega-1}{2}]}
d^{\omega-1}_\ell \cM_\ell+4\omega^3(\omega-1)^3\sum_{\ell=0}^{[\frac{\omega-2}{2}]}
d^{\omega-2}_\ell \cN_\ell\\
& =\sum_{\ell=0}^{[\frac{\omega-1}{2}]}
(\omega-2\ell)d^\omega_\ell [3\cM_\ell+(\omega-2\ell-1) \cN_\ell],
\end{split}
\end{equation*}
since for $\omega$ odd, $\cN_{\frac{\omega-1}{2}}=0$ by definition.
We substitute the value of $\cN_\ell$, given by \eqref{sym ric} and obtain
\begin{multline}\label{111}
(\omega+3)\{6\omega^3 \cM^{\omega-1}+4\omega^3(\omega-1)^3
\cN^{\omega-2}\}=\\
\sum_{\ell=0}^{[\frac{\omega-1}{2}]}d^\omega_\ell (\omega+3)
\{ -2 \cT_\ell-(\omega-2\ell)\cM_\ell\}.
\end{multline}
If $\omega$ is even, then $\cT_{\frac{\omega}{2}}=0$, by \eqref{sym ric}. Thus
\begin{gather}
(\omega+2)\cT^{\omega} =
\sum_{\ell=0}^{[\frac{\omega-1}{2}]}
d^\omega_\ell (\omega+2)
\cT_\ell.\label{222}\\
 2\omega^3(\omega+1)(\cT^{\omega-1}-\cM^{\omega-1}) =\sum_{\ell=0}^{[\frac{\omega-1}{2}]}d^\omega_\ell  (\omega+1)
(\omega-2\ell)(\cT_\ell-\cM_\ell).\label{333}
\end{gather}

The sum of the right hand sides of \eqref{111}, \eqref{222} and \eqref{333}, denoted by $\cI_2$, 
is given by 
\begin{equation*}
\cI_2=\sum_{\ell=0}^{[\frac{\omega-1}{2}]}d^\omega_\ell 
\{[(\omega+1)(\omega-2\ell)-(\omega+4)]
\cT_\ell-2 (\omega-2\ell)(\omega+2)\cM_\ell \}.
\end{equation*}
Substituting the value of $\cM_\ell$ given by \eqref{sym ric} and using the inequality 
of Lemma \ref{lpos}, it follows that  
\begin{equation}\label{I2}
 \cI_2\geq \sum_{\ell=0}^{[\frac{\omega-1}{2}]}2\ell d^\omega_\ell  \cT_\ell.
\end{equation}
For the computation of $\cI_1$, we proceed similarly as above. By equalities \eqref{T ind} 
and \eqref{T1,1} we have
\begin{gather*}
-2\omega^2(\omega-1)(\omega^2+4\omega+5)\cT_1^{\omega-2}=
-\sum_{\ell=1}^{[\frac{\omega-1}{2}]}2\ell(\omega^2+4\omega+5)
d^{\omega}_{\ell} \cT_\ell, \\
4\omega^3(\omega+2)(\omega-1)^3 \cT_{1,1}^{\omega-2}=
\sum_{\ell=1}^{[\frac{\omega-1}{2}]}4
\ell^2(\omega+2) d^{\omega}_{\ell} \cT_\ell,\\
8\omega^3(\omega^2-4)(\omega-1)^2(\cT^{\omega-3}_1-\cM^{\omega-3}_1)=
 \sum_{\ell=1}^{[\frac{\omega-1}{2}]}
4\ell(\omega+2)(\omega-2\ell)d^{\omega}_{\ell} (\cT_\ell-\cM_\ell).
\end{gather*}
Taking the sum of the last three equalities, we obtain
\begin{multline*}
\cI_1=\sum_{\ell=1}^{[\frac{\omega-1}{2}]}
2\ell d^{\omega}_{\ell} \bigl\{[\omega(\omega-2\ell)-4\ell-5 ]\cT_\ell-
2(\omega-2\ell)(\omega+2)  \cM_\ell\bigr\}.
\end{multline*}
Hence
\begin{equation}\label{I1}
\cI_1\geq -\sum_{\ell=1}^{[\frac{\omega-1}{2}]}2\ell d^\omega_\ell \cT_\ell.
\end{equation}

By \eqref{I2} and \eqref{I1}, we conclude that $\cI>4(\omega+1)(\cI_1+\cI_2)\geq 0$.
\end{proof}

\renewcommand{\theequation}{A-\arabic{equation}}    
  \setcounter{equation}{0}  

\vspace*{0.4cm}

\appendix

\section{}\label{det comp}


 {

\subsection{Proof of the equalities $\cA_k=\cB_k$}\label{A=B}

In the following computation, the components of the Ricci tensor $\Ric$ are 
denoted by $R_{ij}$.
Here we give some useful formulas for the computation of the $\cA_k$'s:
\begin{equation}\label{bianchi}
 \nabla_i R_{ij}=\frac{1}{2}\nabla_j \scal
\end{equation}
\begin{equation}\label{der-r1}
 \nabla_a R_{iabj}=-\nabla_b R_{ij}+\nabla_j R_{bi}  \text{ (by 2$^{\rm{nd}}$ Bianchi)}
\end{equation}
\begin{equation}\label{der-r2}
 \nabla_b R_{iabj}=-\nabla_a R_{ij}+\nabla_i R_{aj} 
\end{equation}
\begin{equation}\label{lapl}
 \Delta R_{iabj}=\nabla_{bi} R_{aj}+\nabla_{aj} R_{bi}-\nabla_{ij} R_{ab}-\nabla_{ab} R_{ij}. \text{ (by 2$^{\rm{nd}}$ Bianchi)}
\end{equation}

\begin{equation}\label{inv}
\begin{split}
\cA&:=\nabla_c R_{iabj}\cdot\nabla_c R_{ibaj}=-\nabla_c R_{iabj}\cdot(\nabla_c R_{baij}+\nabla_c R_{aibj})\\
&=|\nabla \Riem|^2-\cA,\quad\text{yielding } \cA=\frac{1}{2}|\nabla \Riem|^2.
\end{split}
\end{equation}
\begin{equation}\label{inv2}
\begin{split}
\cB&:=\nabla_c R_{iabj}\cdot\nabla_b R_{iacj}=-\nabla_c R_{iabj}\cdot(\nabla_j R_{iabc}+\nabla_c R_{iajb})\\
& =|\nabla \Riem|^2-\cB,\quad\text{yielding } \cB=\frac{1}{2}|\nabla \Riem|^2.
\end{split}
\end{equation}
$\cA_1:=\cT^\omega$.\\
$\hspace{-0.4cm}\cA_{2}:=-\tr \Sym \nabla_{I'c} R_{ij}\nabla_{J'b} R_{ibcj}\hspace{-0.2cm}\overset{\eqref{der-r1}}{=}\hspace{-0.2cm}\tr \Sym\nabla_{I'c}R_{ij}\nabla_{J'}(\nabla_c R_{ij}-\nabla_j R_{ci})=\cT^{\omega-1}-~\cM^{\omega-1}$.\\
$\cA_{3}:=-\tr \Sym\nabla_I R_{ij}\nabla_{J''bc} R_{ibcj}\overset{\eqref{der-r2}}{=}
-\tr \Sym\nabla_I R_{ij}\nabla_{J''b}(-\nabla_b R_{ij}+\nabla_i R_{bj})=~\cT^{\omega-2}_{1}$.\\
$\cA_{4}=-\tr \Sym\nabla_{I'b}R_{ij}\nabla_{J'c} R_{ibcj}\overset{\eqref{der-r2}}{=}\tr \Sym\nabla_{I'b} R_{ij}\nabla_{J'}(\nabla_b R_{ij}-\nabla_i R_{bj})=\cA_{2}$.\\
$\cA_{6}:=\cR^\omega$.\\
$\cA_{7}:=\tr {\Sym}\nabla_{I'b}R_{iabj}\nabla_{J'c}R_{iacj}\overset{\eqref{der-r2}}{=}2(\cT^{\omega-1}-\cM^{\omega-1})$.\\
$\cA_{8}:=\tr {\Sym}\nabla_{I''cb}R_{iabj}\nabla_{J}R_{iacj}\overset{\eqref{der-r2}}{=}2\cB_5$.\\
$\cA_{9}:=\tr {\Sym}\nabla_{I'c}R_{iabj}\nabla_{J'b}R_{iacj}\overset{\eqref{inv2}}{=}\frac{1}{2}\cR^{\omega-1}$.\\
$\cA_{10}:=\tr {\Sym}\nabla_{I}R_{iabj}\nabla_{J}R_{ibaj}\overset{\eqref{inv}}{=}\frac{1}{2}\cR^\omega$.\\
$\cA_{11}:=\tr {\Sym}\nabla_{I'b}R_{iabj}\nabla_{J'c}R_{icaj}\overset{\eqref{der-r2}}{=} \cB_2$.\\
$\cA_{12}:=\tr {\Sym}\nabla_{I''bc}R_{iabj}\nabla_{J}R_{icaj}\overset{\eqref{der-r2}}{=} \cB_5$.\\
$\cA_{13}:=\tr {\Sym}\nabla_{I'c}R_{iabj}\nabla_{J'b}R_{icaj}=\tr {\Sym}\nabla_{I'}(\nabla_aR_{icbj}-\nabla_iR_{acbj})\nabla_{J'b}R_{icaj}=\cA_9-\cA_{13}$. Hence $\cA_{13}=\frac{1}{4}\cR^{\omega-1}$.\\
$\cA_{14}:=\tr {\Sym}\nabla_{I}R_{iabj}\nabla_{J''bc}R_{icaj}\overset{\eqref{der-r1}}{=}\cB_5$.\\
$\cA_{15}:=\tr {\Sym}\nabla_{I}R_{iabj}\nabla_{J''ac}R_{icbj}\overset{\eqref{der-r1}}{=}2\cB_5$.\\
$\cA_{16}:=\tr {\Sym}\nabla_{I'c}R_{iabj}\nabla_{J'a}R_{icbj}\overset{\eqref{inv2}}{=}\frac{1}{2}\cR^{\omega-1}$.\\
$\cA_{17}:=\tr {\Sym}\nabla_{I'a}R_{iabj}\nabla_{J'c}R_{icbj}\overset{\eqref{der-r1}}{=}2(\cT^{\omega-1}-\cM^{\omega-1})$.\\
$\cA_{18}:=\tr \Sym \nabla_{I'''ab} R_{iabj} \nabla_{J''cd} R_{icdj}\overset{\eqref{der-r1}}{=}\tr \Sym\nabla_{I''}\Delta R_{ij} \nabla_{J''}\Delta R_{ij}=\cT^{\omega-2}_{1,1}$.\\
$\cA_{20}:=\tr \Sym\nabla_{I'''abc}R_{iabj}\nabla_{J'd} R_{icdj}
\overset{\eqref{der-r2}}{=}\tr \Sym\nabla_{I'''c}\Delta R_{ij} \nabla_{J'} (\nabla_c R_{ij}-\nabla_i R_{cj})$\\
$= \cT^{\omega-3}_{1}-\cM^{\omega-3}_{1}.$\\
$\cA_{21}:=\tr \Sym\nabla_{I'''dab} R_{iabj}\nabla_{J'c} R_{icdj}
\overset{\eqref{der-r1}}{=}\tr \Sym\nabla_{I'''d}\Delta R_{ij} \nabla_{J'} (\nabla_d R_{ij}-\nabla_j R_{di})=~\cT^{\omega-3}_{1}-\cM^{\omega-3}_{1}$.\\
$\cA_{23}:=\nabla_{I'd} \Delta R_{iabj} \nabla_{J'''abc}R_{icdj}\overset{\eqref{der-r1}}{=} 
-\nabla_{I'd}  R_{iabj} \nabla_{J'''dab} R_{ij}\overset{\eqref{der-r2}}{=}\cA_{22}$.\\
$\cA_{24} :=\tr \Sym\nabla_{I''ab} R_{icdj}\nabla_{J''cd}R_{iabj}=\tr \Sym\nabla_{I''ac}R_{ibdj} \nabla_{J''cd}R_{iabj}\overset{\eqref{inv2}}{=}  \frac{1}{2}\cR^{\omega-2}$.\\
$\cA_{25}  :=\tr \Sym\nabla_{I''da} R_{iabj}\nabla_{J''bc}R_{icdj} \overset{\eqref{der-r1}}{=}\cT^{\omega-2}-\cM^{\omega-2}$.\\
$\cA_{26} := \tr \Sym\nabla_{I''ca} R_{iabj} \nabla_{J''bd} R_{icdj}\overset{\eqref{der-r1}, \eqref{der-r2}}{=}\cT^{\omega-2}-2\cM^{\omega-2}+\cN^{\omega-2}$.\\
$ \cA_{27}  :=\tr \Sym\nabla_{I''cb} R_{iabj} \nabla_{J''ad} R_{icdj} \overset{\eqref{der-r2}}{=} \cT^{\omega-2}-\cM^{\omega-2}$.\\

For the remaining terms $\cA_{5}$, $\cA_{19}$ and $\cA_{22}$, the computation is done by induction, by introducing the following sequence  $\cU_{\omega-\beta} := -\tr {\Sym}\nabla_{I_\beta bc}\nabla_{K_\beta}\Ric_{ij}\nabla_{J_\beta}\nabla_{K_\beta} R_{ibcj}$ for $1\leq\beta\leq\omega-2$ and  $\cU_\omega :=\cA_5$, where $K_\beta$, $I_\beta$ and $J_\beta$ are   multi-indices sets of cardinalities $\beta$, $\omega-\beta-2$ and $\omega-\beta$ respectively.  The induction formula is given by
\begin{multline*}
\cU_{\omega-\beta}=2(\omega-\beta-1)(\omega-\beta-2)\cU_{\omega-\beta-1}-\\
2(\omega-\beta-1)^2\tr {\Sym}\nabla_{I_\beta bc}\nabla_{K_\beta}\Ric_{ij}\nabla_{J_{\beta-2}}\nabla_{K_{\beta}}\Delta R_{ibcj}.
\end{multline*}
Using \eqref{lapl}, we have $\tr {\Sym}\nabla_{I_\beta bc}\nabla_{K_\beta}\Ric_{ij}\nabla_{J_{\beta-2}}\nabla_{K_{\beta}}\Delta R_{ibcj}=\cT^{\omega-\beta-2}-2\cM^{\omega-\beta-2}+\cN^{\omega-\beta-2}$.
By induction on $\beta$, we prove that
\begin{equation}\label{A_5}
\begin{split}
\cA_{5} & =\cU_\omega=(\omega-1)!(\omega-2)!\sum_{k=0}^{\omega-2}2^{\omega-k-1}\frac{k+1}{(k!)^2}(\cT^{k}-2\cM^{k}+
\cN^{k})\\
& \overset{\eqref{T ind},\eqref{comb4}}{=}\sum_{\ell=0}^{[\frac{\omega-2}{2}]}e^{\omega}_\ell(\cT_\ell-2\cM_\ell+
\cN_\ell).
\end{split}
\end{equation}

We have $\cA_{22}:= \tr \Sym\nabla_{I'c}R_{iabj} \nabla_{J'''abd}R_{icdj}\overset{\eqref{der-r2}}{=} 
-\tr \Sym\nabla_{I'c} R_{iabj} \nabla_{J'''abc} R_{ij}$. 
By contracting an index in $\cA_5$, which corresponds to a covariant derivative of the Riemann tensor and using the last equality, we obtain 
$\cA_5=2(\omega-1)^2\cB_{26}+2(\omega-1)(\omega-2)\cA_{22}$.
Therefore, by \eqref{A_5}, we obtain 
$\cA_{22} =\sum_{\ell=0}^{[\frac{\omega-3}{2}]}e^{\omega-1}_\ell(\cT_\ell-2\cM_\ell+
\cN_\ell)$.\\

We have $\cA_{19}:=\tr \Sym\nabla_{I''''abcd} R_{iabj}\nabla_J R_{icdj}=-\tr \Sym\nabla_{I''''cd}\Delta R_{ij} \nabla_J R_{icdj}.$
By contracting an index in $\cA_5$, which corresponds to a covariant derivative of the Ricci tensor and using the last equality, we obtain 
$\cA_{5}=2(\omega-1)(\omega-3)\cA_{19}+2\omega(\omega-1)\cA_{22}.$
We deduce that
\begin{equation*}
\cA_{19}= \sum_{\ell=0}^{[\frac{\omega-2}{2}]}\frac{\ell e^{\omega}_\ell}{(\omega-1)(\omega-2)(\omega-3)}(\cT_\ell-2\cM_\ell+
\cN_\ell).
\end{equation*}


\subsection{Combinatorics Formulas}\label{comb}

The following identities hold for any nonnegative integer $\omega$
\begin{gather}
 (\omega+2)\sum_{k=\omega}^{2\omega}(k+1)\binom{k}{\omega}=(\omega+3)^2C(\omega),\label{comb1}\\  
 (\omega+2)\sum_{k=\omega}^{2\omega-1}(k+1)(2\omega-k)\binom{k}{\omega}=\omega(\omega+3)C(\omega),\label{comb2}
\end{gather}
where $C(\omega)=(\omega+1)^2(\omega+2)^2(2\omega+2)![(\omega+3)!]^{-2}$.

\begin{proof}
Identities \eqref{comb1} and \eqref{comb2} can be written in a simpler way  
which can be viewed as special cases (for $n=2\omega$) of the 
following combinatorial identities:
\begin{gather}\label{comb4}
 \sum_{k=\omega}^{n-1}\binom{k}{\omega}=\binom{n}{\omega+1},\quad \sum_{k=\omega}^{n-2}(n-k-1)\binom{k}{\omega}=\binom{n}{\omega+2},
\end{gather}
for all integers $n\geq \omega$, respectively $n\geq \omega+2$.

These identities follow by counting in a different way the number of combinations. The first identity is obtained by counting the number of subsets with $(\omega+1)$ elements out of $n$ elements in the following way: the sets are separated with respect to their largest element. For each $\omega\leq k\leq n-1$, $\binom{k}{\omega}$ counts the subsets of $(\omega+1)$-elements whose largest element is $k+1$.

Similarly, the second identity follows by counting $\binom{n}{\omega+2}$ as follows: 
the sets are separated with respect to their second largest element.  
For each $\omega\leq k\leq n-2$, $(n-k-1)\binom{k}{\omega}$ is the number of 
subsets with $\omega+2$ elements whose second largest element is $k+1$. Indeed, if the second largest element is $k+1$, the others $\omega$ elements of the set which are smaller must form a subset of $\omega+k$ and the largest element may be any of the remaining $n-k-1$ elements.

\end{proof}

\renewcommand\indexname{Notation Index}
\begin{theindex}\label{notation}
{\bfseries A}\nopagebreak

  \item $\cA_k$ \dotfill \textbf{10}

  \indexspace
{\bfseries B}\nopagebreak

  \item $\cB_k$ \dotfill \textbf{10}

  \indexspace
{\bfseries C}\nopagebreak

  \item $C(\omega)=\frac{(\omega+1)^2(\omega+2)^2(2(\omega+1))!}{[(\omega+3)!]^{2}}$ \dotfill 
		\textbf{3}
  \item $c_1,\ldots,c_6$ \dotfill \textbf{10}

  \indexspace
{\bfseries D}\nopagebreak

  \item $d^{\gamma}_{\ell}=\frac{ 2^{\gamma-2\ell}(\gamma !)^3}{(\gamma-2\ell)!(\ell !)^2}$ \dotfill 
		\textbf{9}

  \indexspace
{\bfseries E}\nopagebreak

  \item $e^{\gamma}_{\ell}=\frac{2^{\gamma-2\ell-2}\gamma !(\gamma-1)!(\gamma-2)! }{(\gamma-2\ell-2)!(\ell+1)!\ell !}$ \dotfill 
		\textbf{9}

  \indexspace
{\bfseries I}\nopagebreak

  \item $\cI$ \dotfill \textbf{12}
  \item $\cI_1$, $\cI_2$ \dotfill \textbf{13}
  \item $I$, $J$ multi-indices sets $\#I=\#J=\omega$ \dotfill 
		\textbf{10}
  \item $I'$, $I''$ multi-indices sets $\#I'+1=\#I''+2=\omega$ \dotfill 
		\textbf{10}

  \indexspace
{\bfseries M}\nopagebreak

  \item $\cM_\ell=\nabla_{K}\nabla_{a}\Delta^\ell\Ric_{bc}\cdot \nabla_{K}\nabla_c\Delta^\ell\Ric_{ab}$ \dotfill 
		\textbf{3}
  \item $\cM^{\omega-1-\beta-2\ell}_\ell=\tr \underset{I' \cup J'}{\Sym}  \nabla_{I'Kp} \Delta^\ell \Ric_{ij}  \nabla_{J'Ki}\Ric_{pj}$ \dotfill 
		\textbf{8}

  \indexspace
{\bfseries N}\nopagebreak

  \item $\cN_\ell=\nabla_{K'}\nabla_{cd}\Delta^\ell\Ric_{ab} \nabla_{K'}\nabla_{ab}\Delta^\ell\Ric_{cd}$ \dotfill 
		\textbf{3}
  \item $\cN^{\omega-2-\beta-2\ell}_\ell=\tr \underset{I'' \cup J''}{\Sym}  \nabla_{I''Kpq} \Delta^\ell \Ric_{ij}  \nabla_{J''Kij}\Ric_{pq}$ \dotfill 
		\textbf{8}

  \indexspace
{\bfseries O}\nopagebreak

  \item $\omega=\inf \{k\in \mathbb N : |\nabla^k Weyl_g(x)|\neq0\}$ \dotfill 
		\textbf{2}

  \indexspace
{\bfseries Q}\nopagebreak

  \item $\cQ(R)_{IJabcd}=\nabla_IR_{iabj}\nabla_JR_{icdj}$ \dotfill 
		\textbf{5}

  \indexspace
{\bfseries R}\nopagebreak

  \item $\cR_\ell= \no\nabla^{\omega-2\ell}\Delta^{\ell} \Riem \no^2$ \dotfill 
		\textbf{3}
  \item $\no\nabla^k\Riem\no$ $k$-th derivative norm of $\Riem$ \dotfill 
		\textbf{2}
  \item $\Ric_{ij}=R^\ell_{ i \ell j}$ \dotfill \textbf{1}
  \item $R^q_{jk\ell}X^j=\nabla_k\nabla_\ell X^q-\nabla_\ell\nabla_k X^q$ \dotfill 
		\textbf{1}
  \item $R_{ijk\ell}=g_{iq}R^q_{jk\ell}$ \dotfill \textbf{1}

  \indexspace
{\bfseries S}\nopagebreak

  \item $\mathfrak{S}(k)$ symmetric group of $\{1,\ldots,k\}$ \dotfill 
		\textbf{2}
  \item $\cS =  \tr \underset{\substack{K\cup \{a,b\}, \#K=2\omega+2}}{\Sym} \nabla_{K}\Ric_{ab}$ \dotfill 
		\textbf{5}
  \item $\cS_1=   \tr \underset{K,\#K=2\omega+2}{\Sym}\nabla_{K}\scal$ \dotfill 
		\textbf{5}
  \item $\cS_2=   \tr \underset{\substack{I\cup J\cup \{a,b\},\\ \#I=\#J=\omega}}{\Sym}\nabla_{I}\Ric_{aj}\cdot\nabla_{J}\Ric_{bj}$ \dotfill 
		\textbf{5}
  \item $\cS_3=   \tr \underset{\substack{I\cup J\cup \{a,b\},\\ \#I=\#J=\omega}}{\Sym}\nabla_{I}\Ric_{aj}\cdot\nabla_{J}\Ric_{bj}$ \dotfill 
		\textbf{5}
  \item $\cS_4= \tr\hspace{-0.3cm} \underset{\substack{I\cup J'\cup \{a,b,c\},\\ \#I=\#J'+1=\omega}}{\Sym} \nabla_{I}R_{iabj}\cdot\nabla_{J'}\nabla_i\Ric_{cj}$ \dotfill 
		\textbf{5}
  \item $\Sym T_{p_1\ldots p_k}= \underset{\sigma\in\mathfrak{S}(k)}{\sum} T_{p_{\sigma(1)}\ldots p_{\sigma(k)} }$ \dotfill 
		\textbf{2}

  \indexspace
{\bfseries T}\nopagebreak

  \item $\cT_\ell=\no\nabla^{\omega-2\ell}\Delta^\ell\Ric \no^2$ \dotfill 
		\textbf{3}
  \item $\cT^{\omega-\beta-2\ell}_\ell=\tr \underset{I \cup J}{\Sym}  \nabla_{IK} \Delta^\ell \Ric_{ij}  \nabla_{JK}\Ric_{ij}$ \dotfill 
		\textbf{8}
  \item $\cT^{\omega-2}_{1,1}=\tr \underset{I''\cup J''}{\Sym}\nabla_{I''}\Delta\Ric_{ij}\nabla_{J''}\Delta\Ric_{ij}$ \dotfill 
		\textbf{8}
  \item $\cT^{\omega-\beta}=\cT^{\omega-\beta}_0$ \dotfill \textbf{8}
  \item $\cT^{0}_\ell=(2\ell)\fac\cT_\ell$ \dotfill \textbf{8}
  \item $\tr\Sym T= g^{p_1p_2}\ldots g^{p_{k-1}p_k}\Sym T_{p_1\ldots p_k}$ \dotfill 
		\textbf{2}

\end{theindex}

\bibliographystyle{amsplain}
\bibliography{bibliographie}



\end{document}